\documentclass{article}
\usepackage[utf8]{inputenc}

\usepackage{graphicx}
\usepackage{listings}
\usepackage[ruled,vlined]{algorithm2e}
\usepackage{amssymb}

\usepackage{amsmath}
\usepackage{amsthm}
\usepackage[english]{babel}

\newtheorem{theorem}{Theorem}
\newtheorem{lemma}{Lemma}
\newtheorem{cor}{Corollary}
\newtheorem{defn}{Definition}
\newtheorem{conjecture}{Conjecture}

\begin{document}

\begin{center}
    \Large
    \textbf{Homotopy and Homology at Infinity and at the Boundary}
     
    \vspace{0.4cm}
    \large
    \textbf{Mohammed Barhoush}
    
    \vspace{0.8cm}
    \textbf{Abstract}
\end{center}
In this paper we study the relationship between the homology and homotopy of a space at infinity and at its boundary. Firstly, we prove that if a locally connected, connected, $\delta$-hyperbolic space that is acted upon geometrically by a group has trivial homotopy at infinity then the first Čech homotopy group is trivial. Secondly, we prove that if a hyperbolic group on a finite field has trivial $i^{th}$ homology at infinity then the boundary of the group has trivial $i^{th}$ Steenrod homology. This result turns out to be important in addressing an open problem related to Cannon's conjecture.

\section{Introduction}

Homology and homotopy theories have proven very useful in helping mathematicians classify and understand different spaces. It is why basic homology and homotopy theory has expanded to include new variations such as Čech homology and Steenrod homotopy. In some cases these variations prove easier and more useful to calculate than the original groups. One of the important applications of these groups is to Cannon's conjecture which states:

\begin{conjecture}
(Cannon's conjecture) Let $G$ be a hyperbolic group. If $\partial G\cong S^2$ then $G$ acts geometrically on the hyperbolic space $\mathbb{H}^3$. 
\end{conjecture}

Cannon first proposed this conjecture in an attempt to prove the Thurston Hyperbolization Theorem. Even though the Thurston hyperbolization theorem has already been proven, Connan's conjecture remains a highly important open problem in geometric group theory. This is mainly because it would imply that a closed, irreducible, 3-manifold whose fundamental group is hyperbolic is hyperbolic. This result would be a great leap in our understanding of hyperbolic spaces.

Special cases of Cannon's conjecture have been proven. For instance, Connan's conjecture holds true for Coxeter groups. The most relevant cases of Connan's conjecture have been proven B. Beeker and N. Lazarovich in \cite{beeker} and by V. Markovic in \cite{main}. Combining these two recent papers we have the following result:

\begin{theorem}
\label{intro}
Let $G$ be a hyperbolic group then the following are equivalent
\begin{enumerate}
\item
$G$ acts geometrically on $\mathbb{H}^3$.
\item
$\partial G\cong \mathbb{S}^2$ and $G$ contains 'enough' quasi-convex surface subgroups.
\item
$G$ is one ended, has vanishing first cohomology over $\mathbb{Z}/2$ at infinity and contains 'enough' quasi-convex co-dimension 1 surface subgroups.

\end{enumerate}
\end{theorem}

In B. Beeker and N. Lazarovich's paper the following open question was posed at the end.

\begin{conjecture}
\label{open problem}
If $G$ is one ended, has vanishing first cohomology over $\mathbb{Z}/2$ at infinity and any two points on the boundary of $G$ are separated by a Jordan curve then $\partial G\approx\mathbb{S}^2$.
\end{conjecture}

One of the results proved in this paper (Corollary \ref{SH}) tells us that the vanishing first cohomology over $\mathbb{Z}/2$ at infinity can be implies vanishing first Steenrod homology of the boundary. This is useful because to translate all the conditions of the conjecture to boundary conditions which should be easier to deal with. In other words with the results of this paper we have shown that Conjecture \ref{open problem} is implied by the following conjecture. 

\begin{conjecture}
\label{mines}
If $X=\partial G$ is a connected space such that any two points are separated by a Jordan curve, and $H^{st}_1(X)=0$ then $X\approx\mathbb{S}^2$.
\end{conjecture}

In fact this is a similar result to the following theorem by P. Papazoglou \cite{Papa}.

\begin{theorem}
Let $X$ be a locally compact, geodesic metric space. If $X$ is 1-ended, simply-connected, and any two points are separated by a f-line then $X$ is homeomorphic to the plane.
\end{theorem}

Every condition is satisfied by $G$ but the critical difference is that in Papazoglou's theorem we have that $\pi_1(X)=0$ while in the conjecture we have $H^{st}_1(X)=0$. There is a strong connection between homotopy and homology. This can be seen in Hurewicz theorem (page 338 in \cite{Hatcher}) which states that the first homology group is the abelianization of the first homotopy group. Therefore it is possible that we can modify Papazoglou's proof and obtain a proof for Conjecture \ref{mines} and thus a proof for the open problem proposed by B. Beeker and N. Lazarovich.

This paper is structured as follows. An overview of the different variants of homotopy and homology groups is given in Section \ref{HT}. Afterwards in Section \ref{CH}, it is shown that the natural map from the first Čech homotopy to the first homotopy at infinity is injective. Finally in Section \ref{SH}, it is shown that if the homology at infinity is trivial then the Steenrod homology of the boundary is trivial. 
\section{Preliminaries: Homology and Homotopy at Infinity}
\label{HT}

In this section we will give a brief over view of Čech, Steenrod and locally finite homology theories (based on \cite{remarks on steenrod}) as well as Čech homotopy theory (based on \cite{cech homotopy}). This will prepare the reader for the theorems in the next sections. 

\subsubsection{Simplical homology}

Let $X$ be a $\Delta$-complex. We define $C_n(X)$ to be the abelian group with basis consisting of the $n$-simplices in $X$. Any element in $C_n(X)$ is called an $n$-chain and can be written as a finite sum of $n$-simplices. 

For any $n$-simplex $s^n=[v_0,...,v_n]$, we denote the  $(n-1)$-simplex boundary of $s^n$ that is made up of every vertex excluding $v_i$ by $s^n|_{[v_0,..,\hat{v}_i,..,v_n]}$.

We can define the $n$-boundary homomorphism $\partial_n :C_n\rightarrow C_{n-1}$ that sends an $n$-simplex $s^n$ to its boundary through the following formula. 
$$\partial_ns^n=\sum_{i=1}^n(-1)^is^i|_{[v_1,..,\hat{v}_i,..,v_n]}$$
This defines the homomorphism since the $n$-simplices form a basis of $C_n(X)$. It is easy to show that $\partial_n \circ \partial_{n+1}=0$, i.e. $Im(\partial_{n+1})<Ker(\partial_n )$.

We now define the $n^{th}$ homology group as
$$H_n(X)=Ker(\partial_n )/Im(\partial_{n+1})$$
We call elements of $Ker(\partial_n )$ $n$-cycles and elements of $Im(\partial_{n+1})$ $n$-boundaries.

We can generalize homology groups for any space $X$ by simply replacing $n$-simplices with continuous maps from a $n$-simplex to $X$. The resulting groups are called singular homology groups.

\subsubsection{Locally Finite Homology}

Note that the $n$-chains defined above must be finite sums. However this means that the homology group are not giving us information about the infinite sums of simplices. This turns out to be a problem as many results only hold true when we take into account the infinite chains (for instance Theorem \ref{steenrod}). This motivated the development of locally finite homology.

\begin{defn}
Let $X$ be a locally compact space with the restriction that every compact set $K\subset X$ meets the image of only finitely many simplices from any chain. We define the \textit{locally finite singular homology} of $X$, $H^{LF}_*(X)$, to be the homology theory that includes infinite chains. In other words, an $n$-chain in $C_n^{LF}(X)$ can be an infinite sum of $n$-simplices. 
\end{defn}

\subsubsection{Čech Homology}

We will now briefly introduce the notion of Čech homology. Let $\mathcal{U}=\{U_{\alpha}\}$ be an open cover of a space $X$ where $\alpha$ is some indexing set. We define the \textit{nerve} $\mathcal{N(U)}$ to be the simplical complex made up of a vertex for each $U_{i}\in \mathcal{U}$ and an $(n-1)$-simplex spanned by $n$ vertices $u_{i_1},...,u_{i_n}$ whenever $\cap_{1\leq k\leq n} U_{i_k}$ is non-empty. 

A cover $\{U_{\alpha \in A}\}$ is said to be a \textit{refinement} of another cover $\{V_{\beta \in B}\}$ if for every $\beta\in B$ there exists $\alpha \in A$ such that $U_{\beta}\subset V_{\alpha}$. We get an ordering on the set of open covers of $X$ where $\mathcal{U}_1<\mathcal{U}_2 $ if and only if $\mathcal{U}_2$ is a refinement of $\mathcal{U}_1$. It is easy to show that this set is a directed quasi-ordered set which means that
\begin{enumerate}
\item
$\mathcal{U}<\mathcal{U}$ 
\item
$\mathcal{U}_1<\mathcal{U}_2<\mathcal{U}_3$ entails that $\mathcal{U}_1<\mathcal{U}$ 
\item
For any two covers $\mathcal{U}_1,\mathcal{U}_2$ there exists a cover $\mathcal{U}_3$ such that $\mathcal{U}_1<\mathcal{U}_3$ and $\mathcal{U}_2<\mathcal{U}_3$.
\end{enumerate} 

Furthermore, it can show that the inverse limit in terms of larger and larger refinements is well defined and independent of the sequence of covers. 

Now let $\mathcal{V}>\mathcal{U}$ then there is a natural simplical map $\mathcal{N(V)}\rightarrow \mathcal{N(U)}$ which is well defined up to homotopy. This map is not unique however the induced map on the homology groups is unique.

\begin{defn}
We define the \textit{$i^{th}$ Čech cohomology }of $X$ to be $lim_{\mathcal{U}} \ H(\mathcal{N(U)},G)$ where the limit is for larger and larger refinements $\mathcal{U}$ of $X$. 
\end{defn}

\subsubsection{Steenrod Homology}

The dual of Čech cohomology is called Steenrod homology which we will now define.

\begin{defn}
A \textit{regular map} $f$ from a complex $A$ to a space $X$ is a map that sends vertices in $A$ to $X$ such that for any $\epsilon>0$, all but finitely many simplices in $A$ have their vertices sent onto sets of diameter $<\epsilon$. 
\end{defn}

\begin{defn}
A \textit{regular $n$-chain} is a tuple $(A,f,s^n)$ where $A$ is a complex, $f$ is a regular function from $A$ to $X$, and $s^n$ is a locally finite $n$-chain in $A$. If $s^n$ is a cycle then we call the tuple a regular $n$-cycle.
\end{defn}

Two regular $n$-cycles $(A_1,f_1,s^n_1)$ and $(A_2,f_2,s^n_2)$ are called homologous if there exists a regular $(n+1)$-cycle $(A,f,s^{n+1})$ such that $A_1$ and $A_2$ are closed subcomplexes of $A$, $f$ agrees with $f_1$ on $A_1$ and with $f_2$ on $A_2$, and $\partial s^{n+1}=\partial s^n_1-\partial s^n_2$. It is clear now how the regular cycles can be used to create homology groups.

\begin{defn}
We define the $n^{th}$ Steenrod homology of $X$, denoted by $H^{st}_n(X)$, as the homology group of \textit{regular $n$-cycles} in $X$. 
\end{defn}

\subsubsection{Cech Homotopy}

One of the interesting things about homotopy is its strong relation to homology. This relation can evidently be seen in the following theorem by Hurewicz (page 338 in \cite{Hatcher}).

\begin{theorem}
Let $X$ be a simply connected space. Then 
$$H_n(X)\cong \pi_n(X)$$
when $\pi_i(X)=0$ for all $i<n$.
\end{theorem}

This theorem only holds for Čech homology groups under certain restrictions like local compactness. We want to find a group that is more closely related to the Čech homology group than singular homotopy groups. There is more than one way to define Čech homotopy but we use the following definition.

\begin{defn}
We define the $n^{th}$ Čech homotopy group of $X$ by
$$\hat{\pi}_n(X)=lim_{\mathcal{U}} \ \pi_n(\mathcal{N(U)})$$
where the limit is taken larger and larger refinements (similar to Čech cohomology).
\end{defn}

\section{Homotopy at Infinity and Čech Homotopy}
\label{CH}
In this section we will prove the following theorem. This part of the paper is not relevant to Cannon's conjecture or Conjecture \ref{open problem}. 

\begin{theorem}
\label{extra}
Let $X$ be a locally connected, connected, $\delta$-hyperbolic space that is acted upon geometrically by a group $G$. If the first homotopy group at infinity of $X$ is trivial then the first Čech homotopy group of the boundary $\partial X$ is trivial.
\end{theorem}

\begin{defn}
The first homotopy at infinity of $X$ based on the geodesic ray $w$ is defined as
$$\pi_1^{\infty}(X,w)=\lim_{i\rightarrow \infty}\pi_1(X\setminus \overline{B}(i),w(i+1))$$
where $\overline{B}(i)$ is the closed ball of radius $i$ around $w(0)$.
\end{defn} 

It can be equivalently defined by replacing the balls with any sequence of compact sets ordered by inclusion but we use the above definition for simplicity.

More explicitly, we have a sequence

$$\pi_1(X\setminus \overline{B}(1),w(2))\xrightarrow{i_1}\pi_1(X\setminus \overline{B}(2),w(3))\xrightarrow{i_2}....$$

where $i_j$ is the composition of the inclusion map $\pi_1(X\setminus \overline{B}(k),w(k+1))\rightarrow \pi_1(X\setminus \overline{B}(k-1),w(k+1))$ and the isomorphism $\pi_1(X\setminus \overline{B}(k-1),w(k+1))\rightarrow \pi_1(X\setminus \overline{B}(k-1),w(k))$ which simply slides the base-point of the loops from $w(k+1)$ to $w(k)$ along $w$.
  
G. Conner and H. Fisher proved a similar result \cite{Conner} to Theorem \ref{extra}.

\begin{theorem}
Let $X$ be a geodesic negatively curved space. Then there is a natural homomorphism between the fundamental group of the boundary of $X$ to its fundamental group at infinity. This homomorphism is an isomorphism when $X$ admits a universal cover and is an injection when $X$ is one-dimensional.
\end{theorem}

Let $\gamma\in \pi_1(\partial G)$ with base point $x_0$. The natural homomorphism referred to in the theorem is the following
$$\gamma \rightarrow \{\gamma_1,\gamma_2,...\}$$
where $\gamma_i$ is the intersection of $\gamma$ with $S_i(x_0)$. The limit of such a sequence is an element of the fundamental group at infinity (more details in \cite{Conner})

Ofcourse a loop in $\hat{\pi}_1(\partial X)$ is made up of a discrete points so it will map to descrete points on $X$ but we can connect these points with geodesic and get loops as well.
This gives us a similar map from $\hat{\pi}_1(\partial X)$ to $\pi_1^{\infty}(X)$. Theorem \ref{extra} will allow us to show that for any hyperbolic space $X$, this map is injective. 

We define $\mathcal{U}_{\epsilon}=\{U_{\alpha}\}$ to be the open cover where $U_{\alpha}=B_{\epsilon}(u_{\alpha})$ for some $u_{\alpha}\in X$. Note that a loop in $\mathcal{N(U_{\epsilon})}$ is made up of a set of points $\{u_1,...,u_n\}$ such that $U_i\cap U_{i+1}\neq \emptyset$.

\begin{defn}
\label{weird}
If $L=\{u_1,..,u_n\}$ is a loop in $\mathcal{N(U_{\epsilon})}$. We say that a set of points $S$, where $\{u_1,..,u_n\}\subset S\subset X$, is an \textit{$\epsilon$-subdivision} of $L$ if there exists $ \delta>0$ and a subset $D\subset (\mathbb{D}^2)^{(0)}$ such that 
\begin{enumerate}
\item
Every ball of radius $\delta$ in $\mathbb{D}^2$ contains a point $d\in D$.
\item
There exists a bijective map $\phi :D\rightarrow S$ such that for any $d_1,d_2\in D$

$$ d(d_1,d_2)\leq 10\delta \implies d(\phi(d_1),d(\phi(d_2))<\epsilon$$
\item 
$$\phi^{-1}(L)\subset \partial \mathbb{D}^2$$
\end{enumerate}
\end{defn}

\begin{lemma}
\label{subdivision}
Let $\mathcal{N(U_{\epsilon})}$ be the nerve of a space $X$ and let $L=\{u_1,...,u_n\}$ be a loop in the nerve. If there exists a $\epsilon$-subdivision of $L$ then $L$ is contractible. 
\end{lemma}

\begin{proof}
Connect any two points in $D$ that are a distance at most $10\delta$ apart with a line. This would partition $D$ into triangles. This is because any ball $B_{10\delta}$ in $\mathbb{D}^2$ must contain at least 3 points and every point is adjacent to every other point which creates a triangle subdivision in $B_{10\delta}$.  Thus we get a triangle subdivision of $D$. Now if $d_1,d_2\in D$ are adjacent then $d(\phi(d_1),d(\phi(d_2))<\epsilon$ so $\phi(d_1)$ and $\phi(d_2)$ form a 1-simplex. In other words triangles are sent to triangles. Note that $L$ is in the image of $\phi$ so $\phi(D)$ gives us a subdivision of $L$ into triangles. But every triangle will span a 2-simplex in $\mathcal{N(U_{\epsilon})}$ since each of the vertices are within a distance of $\epsilon$. Therefore $L$ is contractible. 
\end{proof}

Note that the proof still works if we replace a disk with a shape homeomorphic to a disk. 

For any ray $r\in \partial X$ define $r(t)$ to be the point on the ray that is a distance $t$ from its base point. The following lemma will be useful in the proof of Theorem \ref{extra}.

\begin{lemma}
\label{1}
Let $r,r'$ be two geodesic rays in $X$ with base point $x_0\in X$. If for $t'>\delta log_2(10C)+1+log(\epsilon)$ we have $d(r(t'),r'(t'))< 10C$ then $d_a(r,r')< \epsilon$. 
\end{lemma}

Here $C>\delta$ is a constant we shall explain later.

\begin{proof}
Clearly there exists $t_0>0$ such that $d(r(t_0),r'(t_0))=\delta$. By Proposition 1.25 in \cite{L} we know that for any $t>t_0$, $d(r(t),r'(t))> 2^{\frac{t-t_0-1}{\delta}}$. This means that for $t=t_0+\delta log_2{10C}+1$ we get that $d(r(t),r'(t))>10C$.

We know that $d(r'(t'),r(t'))<10C$ so $t_0\geq t'-\delta log_2{2C}+1\geq log (\epsilon)$. Therefore $d_a(r_1,r_1')=e^{t_0}< \epsilon$.
\end{proof}

We are now ready to prove Theorem \ref{extra}.

\begin{proof}

Let $\epsilon >0$ and $\mathcal{U}_{\epsilon}=\cup_{i\in I}U_i$ be an open cover of $\partial X$. We can then construct a nerve $\mathcal{N(U_{\epsilon})} $ made up of the open sets $\{U_i\}_{i\in I}$ and their corresponding points $\{u_i\}_{i\in I}$. 

We need to show that any loop in $\mathcal{N(U_{\epsilon})}$ can become contractible by making $\mathcal{U_{\epsilon}}$ finer.  
 
Let $L$ be a loop in $\mathcal{N(U)}$ spanned by the rays $\{r_1,...,r_n\}$ where for $1\leq j\leq n$, ${r_j}$ and ${r_{j+1}}$ ($r_{n+1}=r_{1}$) bound a 1-simplex, i.e are within a distance of $\epsilon$ from each other. 

For any $t>0$, the points $\{r_1(t),...,r_n(t)\}$ are part of a loop, call it $L(t)$, in $X$. This is because $X$ is a geodesic space so we can connect any two consecutive points in this set with a geodesic segment to form a loop. 

\begin{lemma}
There exists $r_0\in \mathbb{N}$ such that for every integer $r>r_0$, every loop $L(t)$, $t\in \mathbb{N}$ and $t>r$, is contractible in $X\setminus \overline{B_r}(x_0)$.
\end{lemma}

\begin{proof}
We have a sequence of loops $\{L(1),L(2),...\}$ getting further and further from $x_0$. These loops are not collapsing to a point since rays eventually diverge. If this lemma is false then for every $r_0\in \mathbb{N}$ there exists $r>r_0$ such that $X\setminus \overline{B_r}(x_0)$ is not simply connected. So clearly the limit of the fundamental group of $\{X\setminus \overline{B_1}(x_0),X\setminus \overline{B_2}(x_0),...\}$ is not trivial which is a contradiction.
\end{proof}

Therefore we can find $r-1>\delta log_2(10C)+1+log(\epsilon)$ such that $L(t)$ is contractible in $X\setminus \overline{B_{r-1}}(x_0)$ for any $t>r-1$.

By Lemma 3.1 in \cite{BM} we know that any ball of radius $C$ intersects a geodesic ray from $x_0$. This is where the constant $C$ comes from. Let $R$ be the set of rays that pass within a distance of $C$ from one of the points in the filling of $L(r)$. 

Now let $\delta=C$ and $\phi$ be the map that sends the a point of intersection of a geodesic ray in $R$ with the filling of $L(r)$ to the geodesic ray itself. We need to check that this map satisfies all the conditions of Definition \ref{weird}. The first condition is satisfied by Lemma 3.1 in \cite{BM}. The second condition is satisfied by Lemma \ref{1}. It is clear that the third condition is satisfied. Therefore by Lemma \ref{subdivision} we know that $L$ is fillable and so the first Cech homology of $\partial X$ is trivial.
\end{proof}
\section{Homology at Infinity and Steenrod Homology}
\label{SH}
In this section we will prove the following theorem.

\begin{theorem}
\label{2}
Let $G$ be a hyperbolic group and let $k$ be a finite field. If $G$ has trivial $i^{th}$ homology over $\textit{k}$ at infinity then $\partial G$ has trivial $i^{th}$ Steenrod homology over $k$. 
\end{theorem}

\begin{defn}
We define the Rips complex $P_t(G)$ of $G$ to be the simplical complex whose vertices are elements of $G$ and any set of vertices $v_1,...,v_n\in G$ span an $n$-simplex if $d(v_i,v_j)<t$ for all $1\leq i,j\leq n$. 
\end{defn}

The useful thing about the Rips complex is that for large enough $t$ , $P_t(G)$ becomes contractible. Set $P(G)=P_t(G)$ for $t$ large enough.

The action of $G$ on itself by left translation can be extended to an action on $P(G)$. Moreover, its is easy to show that $G$ acts geometrically on the Rips complex since it acts geometrically on itself.  Therefore the assumption in the theorem tells us that $H^{\infty}(P(G))=0$. (More details about the Rips complex can be found in Section 9.2 in \cite{gg}).

Now define $B(n)$ to be the subset of $P(G)$ that contains every vertex $v\in G$ that satisfies $d(v,e)<n$, where $e$ is the identity of $G$, and every simplex whose vertices are in $B(n)$. Similarly define $S(n)$ to be the subset of $P(G)$ that contains every vertex $v\in G$ that satisfies $d(v,e)=n$ and every simplex whose vertices are in $S(n)$. Finally, define $\overline{B}(n)$ to be the subset that contains every vertex $v\in G$ that satisfies $d(v,e)\leq n$ and every simplex whose vertices are in $\overline{B}(n)$. We are now ready to prove Theorem \ref{2}.

\begin{proof}
We will begin the proof with the following theorem by M. Bestvina (Proposition 1.5 in \cite{BM}). 

\begin{theorem}
\label{steenrod}
Let $G$ be a hyperbolic group and let $R$ be a ring. Then there is an isomorphism of $RG$-modules
$$H^{st}_i(\partial G)\cong H_{i+1}^{LF}(P(G))$$
\end{theorem}

Therefore it is sufficient to prove that $H_{i+1}^{LF}(P(G))=0$. 

Let $c\in C_{i+1}^{LF}(P(G))$ be an $(i+1)$-cycle. This means $c$ is a sum of simplices so we can write $c=\sum_{s\in S^{i+1}}\alpha(s)s$ where $S^{i+1}$ is the set of all $(i+1)$-simplices in $P(G)$ and $\alpha$ is a function that sends a simplex to its coefficient in $c$. We need to show that $c$ is an $(i+1)$-boundary i.e. it bounds an $(i+2)$-chain. It is well known that a contractible space has trivial simplical homology, therefore if $c$ is a finite sum of simplices then $c$ is a boundary. 

So assume $c$ is an infinite sum of simplices. We define $c\ \cap B(j)=\sum_{s\in S^{i+1}}\mathfrak{i}(s)\alpha(s)s$ where $\mathfrak{i}$ is the function that sends a simplex $s$ to 1 if every vertex of $s$ is in $B(j)$ and 0 otherwise. Similarly for $\overline{B}(j)$ and $S(j)$. Let $n\in \mathbb{N}$ be large enough so that the intersection $c\cap B(n)$ is non-empty. 

\textbf{Claim:} The boundary of $c\cap \overline{B}(n)$ is an $i$-cycle.

By definition, $c\cap \overline{B}(n)$ is a sum of $(i+1)$-simplices in $c$ whose vertices are within a distance of $n$ from the identity. The boundary of $(i+1)$-simplices are $i$-boundaries and thus $i$-cycles. 

It is important to note that these cycles may not lie on $S(n)$. However they have to be a distance of at most $t$ from $S(n)$, where $t$ is the constant we defined when discussing the Rips complex.

Since the $i^{th}$ homology at infinity is trivial we get that for any $n>0$ there exists $m>n+t$ such that every $i$-cycle in $P(G)\setminus \overline{B}(m)$ is will bound an $(i+1)$-chain in $P(G)\setminus \overline{B}(n+t)$. Denote the sum of the $(i+1)$-chains that bound $c\cap\overline{B}(m)$ by $c(m)$. It is important to note that $c(m)$ is outside the ball $B(n+t)$ so it cannot intersect to $(c\cap \overline{B}(n-1))$. The constant $m$ depends on $n$ so write $m=m(n)$.

Consider the finite sum of $(i+1)$-simplices, $c_n=c(m(n))-(c\cap \overline{B}(m(n)))$. It is clear that this is an $(i+1)$-cycle since the boundary of $c(m)$ is equal to the boundary of $c\cap \overline{B}(m)$ so they cancel out. As it is a finite cycle it must be a boundary. Let $b_n$ be an $(i+2)$-cycle such that $\partial_{i+2} b_n=c_n$ where $\partial_{i+2}$ is the $(i+2)$ boundary map of the homology groups. In a similar fashion as $c$, we write $b_j=\sum_{s\in S^{i+2}}\beta_j(s)s$.

So we now have a sequence of 'fillings' that fill larger and larger 'sections' of $c$. We want to show that the limit of these filling will fill $c$. However, $lim_{j\rightarrow \infty} b_j$ is not necessarily well defined. This is because the coefficient of any $(i+2)$-simplex may continuously fluctuate in the sequence $\{b_j\}$ and never converge. However there is a way around this problem. 

Assume $s'$ is an $(i+2)$-simplex in $B(1)$. Since $k$ is finite, any sequence of coefficients $\{\beta_j(s')\}$ must contain a subsequence that is constant. Correspondingly, let $\{b_{j'}\}$ be a subsequence of $\{b_j\}$ where the coefficient of $s'$ are constant. Now take another simplex $s''$ in $B(1)$. We can repeat this process and find a subsequence of $\{b_{j'}\}$ such that the coefficient of $s''$ is constant. Since $P(G)$ is locally finite there are finitely many $(i+2)$-simplices in $B(1)$. Therefore after repeating this process finitely many times we get a subsequence of $\{b_{j}\}$ such that the coefficients of any $(i+2)$-simplex in $B(1)$ is constant. Call this subsequence $\{b_{j_1}\}$. 

Now replace $\{b_j\}$ with $\{b_{j_1}\}$ and repeat this process on $B(2)$. We will get a subsequence of $\{b_{j_1}\}$ such that the coefficients of any $(i+1)$-simplex in $B(2)$ are constant. Denote this subsequence by $\{b_{j_2}\}$. Now consider the following sequence 
$$\{b_{j_j}\}=\{b_{1,1},b_{2,2},...\}$$

\textbf{Claim:} The limit of this sequence is well defined.

Let $s$ be an arbitrary $(i+2)$-simplex. Then there exists $n\in \mathbb{N}$ such that $s\in B(n)$. Note that the sequence $\{b_{j_j}\}_{j\geq n}$ is a subsequence of $\{b_{j_n}\}$ which means that the coefficient of $s$ is constant in $\{b_{j_j}\}_{j\geq n}$. This means that the limit of the coefficients of $s$ is well defined. In other words, the limit of the coefficient of any $(i+2)$-simplex is well defined. Hence the limit of $\{b_{j_j}\}$ is well defined and we denote this limit as $b$.

\textbf{Claim:} $\partial_{i+2} b=c$. 

Assume this is not true. Then for some $(i+1)$-simplex $s^{i+1}$, the coefficient of $s^{i+1}$ in $c$ is not equal to the coefficient of $s^{i+1}$ in $\partial_{i+2} b$. Let $n\in \mathbb{N}$ be such that $s^{i+1}\in B(n-2)$ and let $m$ be an arbitrary integer larger than $n+t$.  Recall that $\partial_{i+2} b_{m_m}=c(m_m)- (c\ \cap \overline{B}(m_m))$. Also recall that the simplices in $c(m_m)$ are outside the ball $B(m_m)$ so they do not contribute to the coefficient of $s^{i+1}$ in $\partial_{i+2}b_{m_m}$. In other words the coefficient of $s^{i+1}$ in $\partial_{i+2} b_{m_m}$ is equal to $\alpha(s^{i+1})$. As this is true for any $m>n+t$ it must also be true for the limit as well. 

Therefore every $(i+1)$-cycle is a boundary and thus $H^{LF}_{i+1}(P(G))=0$. By Theorem \ref{steenrod} we get that $H^{st}_i(\partial G)=0$ and we are done.
\end{proof}

The following is a direct Corollary of this theorem and Lemma \ref{relator}.

\begin{cor}
If a hyperbolic group $G$ has vanishing first cohomology at infinity over $\mathbb{Z}/2$ then $H^{st}_1(\partial G)=0$.
\end{cor}

This implies that the open problem Conjecture \ref{open problem} is implied by Conjecture \ref{mines}. 

\end{document}